\documentclass{article}

\usepackage{amsmath,amssymb}
\usepackage{amsthm}
\usepackage{float}

\usepackage{graphicx}
\usepackage{enumitem}
\usepackage{multicol}
\usepackage{hyperref}
\usepackage{url}

\usepackage{color}

\usepackage{tikz,pgfplots}
\tikzset{
every node/.style={circle, inner sep=2pt}
}
\pgfplotsset{yticklabel style={
        /pgf/number format/fixed,
        /pgf/number format/precision=5
},
scaled y ticks=false}

\usetikzlibrary{matrix,arrows,calc}

\pgfplotsset{
    legend entry/.initial=,
    every axis plot post/.code={%
        \pgfkeysgetvalue{/pgfplots/legend entry}\tempValue
        \ifx\tempValue\empty
            \pgfkeysalso{/pgfplots/forget plot}%
        \else
            \expandafter\addlegendentry\expandafter{\tempValue}%
        \fi
    },
}

\newtheorem{theorem}{Theorem}

\newtheorem{proposition}[theorem]{Proposition}
\newtheorem{corollary}[theorem]{Corollary}

\theoremstyle{definition}

\newcommand{\diag}{\operatorname{diag}}
\newcommand{\SNF}{\operatorname{SNF}}

\newcommand{\SP}{\operatorname{sp}}
\newcommand{\IN}{\operatorname{in}}

\pgfplotsset{compat=1.14}
\begin{document}

\title{Enumeration of cospectral and coinvariant graphs}
\author{Aida Abiad$^{a,b}$ and Carlos A. Alfaro$^c$
\\ \\
{\small $^a$Department of Mathematics and Computing Science} \\
{\small Eindhoven University of Technology, Eindhoven, The Netherlands}\\
{\small $^b$Department of Mathematics: Analysis, Logic and Discrete Mathematics} \\
{\small Ghent University, Ghent, Belgium}\\
{\small {\tt
a.abiad.monge@tue.nl}} \\
{\small $^c$ Banco de M\'exico} \\
{\small Mexico City, Mexico}\\
{\small {\tt
carlos.alfaro@banxico.org.mx}} \\
}
\date{}

\maketitle

\begin{abstract}
We present enumeration results on the
number of connected graphs up to 10 vertices for which there is at least one other graph with the same spectrum (a cospectral mate), or at least one other graph with the same Smith normal form (coinvariant mate) with respect to several matrices associated to a graph.
The present data give some
indication that possibly the Smith normal form  of  the  distance  Laplacian  and  the  signless  distance  Laplacian matrices could be a finer invariant to distinguish graphs in cases where other algebraic invariants, such as those derived from the spectrum, fail. Finally, we show a new graph characterization using the Smith normal form of the signless distance Laplacian matrix. 
\end{abstract}



\section{Introduction}

Spectral graph theory aims to understand to what extent graphs are characterized by their spectra. Starting from the eigenvalues of a matrix associated to a graph, it seeks to deduce combinatorial properties of the graph. For this, we associate a graph $G$ to a matrix $M$ and analyze the eigenvalues of $M$. These eigenvalues are called the \emph{spectrum} of $G$ with respect to the matrix $M$, and its multiset
is denoted by $M$-spectrum(G). $M$-\emph{cospectral graphs} are graphs that share $M$-spectrum. A graph $G$ is \emph{determined by its $M$-spectrum} ($M$-DS for short) if only isomorphic graphs are cospectral with $G$.

In connection with the graph isomorphism problem, it is of interest what fraction of all graphs is uniquely determined by its spectrum. Haemers
conjectured that the fraction of graphs on $n$ vertices with a $M$-cospectral mate tends to zero as $n$ tends to infinity. Numerical data for $n\leq 9$ was given by Godsil and McKay \cite{gm}, for $n = 10, 11$ by Haemers and Spence \cite{HS2004} and for $n=12$ by Brouwer and Spence \cite{bs}. Aouchiche and Hansen \cite{ah} presented a numerical study in which they looked into the spectra of the distance, distance Laplacian and distance signless Laplacian matrices of all the connected graphs on up to 10 vertices. Recently, Pinheiro, Souza and Trevisan \cite{pst2020} provided some numerical evidence that the complementary spectrum of a graph (the \emph{complementary spectrum} of a graph $G$ is the set of the complementary eigenvalues of the adjacency matrix of $G$) distinguishes more graphs than other standard graph spectra. However, the authors also showed that it is hard to compute the complementary spectrum.

The question arises whether it is possible to define a matrix $M$ of $G$ in a (not so sensible) way such that every graph becomes $M$-DS. This is indeed the case, as it was shown in \cite[Section 2.5]{vDH}. However, in this case it is more work to check cospectrality of the matrices
than testing isomorphism. If there would be an easily computable matrix $M$ for which every graph becomes $M$-DS, the graph isomorphism problem would be solved. Hence, when $M$ is one of the commonly used matrices associated to graphs (adjacency, Laplacian, signless Laplacian, normalized Laplacian, distance matrices), one can say that there is not such a matrix $M$ for which all graphs are $M$-DS, since there exist many examples of non-isomorphic graphs that share the same $M$-spectrum. This leaves open the possibility of amplifying spectra with the use of more refined representations for obtaining more faithful graph information.

The main goal of this article is to propose a new way of representing a graph using the Smith normal form (SNF) of certain distance matrices. We  provide some numerical evidence that this new algebraic graph representation may do a better job in distinguishing graphs. For this we first need to recall some definitions.
A matrix $M$ is said to be \emph{equivalent} to $N$ if there exist unimodular matrices $U$ and $V$ with entries in $\mathbb{Z}$ such that $M=UNV$.
The Smith normal form of a integer matrix $M$, denoted by $\SNF(M)$, is the unique diagonal matrix $\diag(f_1,\dots,f_r,0,\dots,0)$ equivalent to $M$ such that $r=rank(M)$ and $f_i|f_j$ for $i<j$.
The \emph{invariant factors} (or \emph{elementary divisors}) of $M$ are the integers in the diagonal of the $\SNF(M)$.
If $M$ is an integer symmetric matrix associated to a graph, then we say that two graphs $G$ and $H$ are $M$-{\it coinvariant} if the SNFs of $M(G)$ and $M(H)$, computed over $\mathbb{Z}$, are the same.
Conivariant graphs were introduced in \cite{vince}. Note that related to the SNF there is the $p$-\emph{rank}, i.e., the rank of the matrix considered over the finite field $\mathbb{F}_p$, which is the number of invariant factors not divisible by $p$. We should note that the $p$-rank has also been used in the literature to distinguish graphs. Just to name a few, the $2$-rank was used joint with the spectrum to characterize symplectic graphs over $\mathbb{F}_2$ \cite{Peeters} and was used to distinguish strongly regular graphs with the same parameters as the symplectic graph \cite{ah2016}. Some relevant $p$-ranks joint with the spectrum was used to characterize distance-regular graphs \cite{Peeters2002}.

In particular, in this work we study if there is a matrix $M$ (say adjacency, Laplacian, signless Laplacian, etc.) whose SNF distinguishes more graphs. Broadly speaking, the
idea is to verify whether the portion of graphs that have a $M$-coinvariant mate is smaller than the portion of graphs having a $M$-cospectral mate for a particular matrix $M$. 
Cospectrality and coinvariancy both play an important role in the famous graph isomorphism problem. Actually, it is not yet known if testing graph isomorphism is a hard problem or not, whether determining whether two graphs are cospectral or coinvariant can be done in polynomial time \cite{SNFP,kannan}. It is also known that testing coinvariancy is experimentally faster than testing cospectrality \cite{akm}. Thus, one can focus on testing isomorphism among coinvariant graphs.

Our results show that the SNF of the signless distance Laplacian matrix provide a way of representing graphs which does a better job than the spectrum in distinguishing them. The distance Laplacian and the signless distance Laplacian matrices have received quite a lot of attention over the last years \cite{ah2014,ah2013,bdhlrsy2019,L2020,DAH2018,np2014}. This article is a sequel to the work by Aouchiche and Hansen \cite{ah}. Numerical data on the number of cospectral and coinvariant graphs is given for several matrices, and we also take the opportunity to correct an earlier value. This paper also complements the work by Haemers and Spence  \cite{HS2004}, Lepović \cite{L1998} and Godsil and McKay \cite{GM1982} on enumerating cospectral graphs.

In particular, we extend the computer enumeration for cospectral graphs of \cite{GM1982}, \cite{L1998}, \cite{HS2004} and \cite{ah} to all connected graphs on at most 10 vertices that have at least a cospectral mate with respect to the distance Laplacian matrix and the signless distance Laplacian matrix. We also enumerate graphs with at most 10 vertices which have a coinvariant mate for several associated matrices. 
Finally, we present a novel way to show a graph characterization using the SNF of the distance Laplacian and signless distance Laplacian matrix, illustrating the power of this new proposed graph invariant.

\section{Enumeration}\label{sec:enumaration}

In order to make our enumeration results for several matrices comparable with the data obtained for the graph distance matrices, we focus on connected graphs. Denote by $\mathcal{G}_n$ the set of connected graphs with $n$ vertices.
Given a connected graph $G$, we will study the following associated matrices: the adjacency matrix $A(G)$, the Laplacian matrix $L(G)$, the distance matrix $D(G)$, the signless Laplacian matrix $Q(G)$, the distance Laplacian matrix $D^L(G)$ and the signless distance Laplacian matrix $D^Q(G)$.

Let $\mathcal{G}^{sp}_n(M)$ be the set of graphs in $\mathcal{G}_n$ which have at least one cospectral mate in $\mathcal{G}_n$ with respect to an associated matrix $M$. Table~\ref{Tab:cospectralALQD} provides the number of cospectral mates of connected graphs with respect to several associated matrices.

\begin{table}[H]
    \centering
	{\small
    \begin{tabular}{ccccccc}
		\hline
        $n$ & 5 & 6 & 7 & 8 & 9 & 10\\
        \hline
        $|\mathcal{G}_n|$ & 21 & 112 & 853 & 11,117 & 261,080 & 11,716,571\\
        \hline
        $|\mathcal{G}^{sp}_n(A)|$ & 0 & 2 & 63 & 1,353 & 46,930 & 2,462,141 \\
        $|\mathcal{G}^{sp}_n(L)|$ & 0 & 4 & 115 & 1,611 & 40,560 & 1,367,215 \\
        $|\mathcal{G}^{sp}_n(Q)|$ & 2 & 10 & 80 & 1,047 & 17,627 & 615,919  \\
        $|\mathcal{G}^{sp}_n(D)|$ & 0 & 0 & 22 & 658 & 25,058 & 1,389,986   \\
        \hline
	\end{tabular}
	}
\caption{Number of connected graphs with at least one cospectral mate for the adjacency, Laplacian, signless Laplacian and distance matrices.}
	\label{Tab:cospectralALQD}
\end{table}

Analogously, let $\mathcal{G}^{in}_n(M)$ be the set of graphs in $\mathcal{G}_n$ which have at least one coinvariant mate in $\mathcal{G}_n$ with respect to an associated matrix $M$. Table 	\ref{Tab:coinvariantALQD} shows the enumeration of $\mathcal{G}^{in}_n(M)$ for several associated matrices.

\begin{table}[H]
    \centering
	{\small
    \begin{tabular}{cccccccc}
		\hline
        $n$ & 4 & 5 & 6 & 7 & 8 & 9 & 10\\
        \hline
        $|\mathcal{G}_n|$ & 6 & 21 & 112 & 853 & 11,117 & 261,080 & 11,716,571\\
        \hline
        $|\mathcal{G}^{in}_n(A)|$ & 4 & 20 & 112 & 853 & 11,117 & 261,080 & 11,716,571 \\
        $|\mathcal{G}^{in}_n(L)|$ & 2 & 8 & 57 & 526 & 8,027 & 221,834 & 11,036,261 \\
        $|\mathcal{G}^{in}_n(Q)|$ & 2 & 11 & 78 & 620 & 7,962 & 201,282 & 10,086,812 \\
        $|\mathcal{G}^{in}_n(D)|$ & 2 & 15 & 102 & 835 & 11,080 & 260,991 & 11,716,249\\
        \hline
	\end{tabular}
	}
\caption{Number of connected graphs with at least one coinvariant mate for the adjacency, Laplacian, signless Laplacian and distance matrices.}
	\label{Tab:coinvariantALQD}
\end{table}

Extensive research has been devoted to understand cospectral graphs, but much less has been dedicated to understand coinvariant mates and its potential to characterize graphs.
The reason for this could be that for matrices $A$, $L$, $Q$ and $D$, there is a large proportion of connected graphs having a $M$-coinvariant mate, as Figure \ref{fig:spectruminvariantALQD} shows. We follow \cite{CRS} in defining the \emph{spectral uncertainty} $\SP_n(M)$ with respect to $M$ as the ratio $|\mathcal{G}^{sp}_n(M)|/|\mathcal{G}_n|$, and the \emph{invariant uncertainty} $\IN_n(M)$ with respect to $M$ as the ratio $|\mathcal{G}^{in}_n(M)|/|\mathcal{G}_n|$.

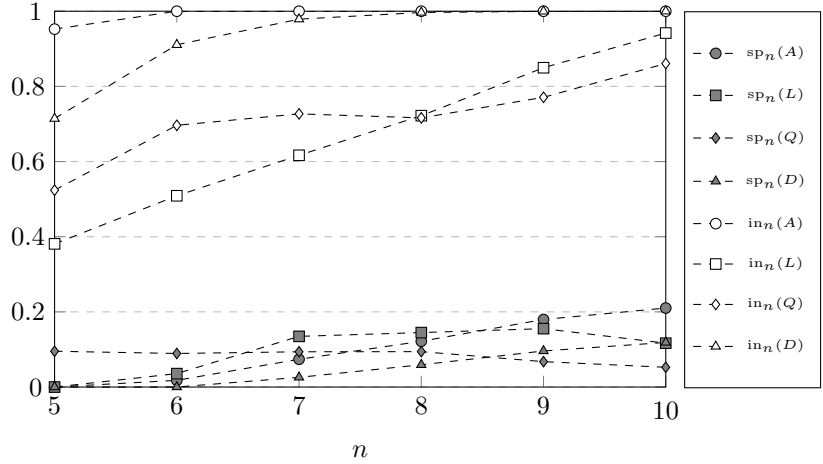
\begin{figure}[H]
    \centering
    \begin{tikzpicture}[trim axis left]
    
    \begin{axis}[
        scale only axis,
        title={},
        xlabel={$n$},
        width=\textwidth-4cm, 
        height=5cm,
        xmin=5, xmax=10,
        ymin=0.0, ymax=1.0,
        xtick={5,6,7,8,9,10},
        legend style ={ 
            row sep=-0.4cm,
            at={(1.03,1)}, 
            anchor=north west,
            draw=black, 
            fill=white,
            align=left
        },
        ymajorgrids=true,
        grid style=dashed,
        legend columns=1
    ]
     
    \addplot[
        dashed,
        every mark/.append style={solid, fill=gray},
        mark=*,
        legend entry={\tiny $\SP_n(A)$}
        ]
        coordinates {
        (5,0.000000000000000)(6,0.0178571428571429)(7,0.0738569753810082)(8,0.121705496087074)(9,0.179753332311935)(10,0.210141772708073)
        };
     
    \addplot[
        dashed,
        every mark/.append style={solid, fill=gray},
        mark=square*,
        legend entry={\tiny $\SP_n(L)$}
        ]
        coordinates {
        (5,0.000000000000000)(6,0.0357142857142857)(7,0.134818288393904)(8,0.144913196006117)(9,0.155354680557683)(10,0.116690710959717)
        };
    
    \addplot[
        dashed,
        every mark/.append style={solid, fill=gray},
        mark=diamond*,
        legend entry={\tiny $\SP_n(Q)$}
        ]
        coordinates {
        (5,0.0952380952380952)(6,0.0892857142857143)(7,0.0937866354044549)(8,0.0941800845551858)(9,0.0675157039987743)(10,0.0525681959337762)
        };
        
    \addplot [
        dashed,
        every mark/.append style={solid, fill=gray},
        mark=triangle*,
        legend entry={\tiny $\SP_n(D)$}
        ]
        coordinates {
        (5,0.000000000000000)(6,0.000000000000000)(7,0.0257913247362251)(8,0.0591886300260862)(9,0.0959782442163322)(10,0.118634197667560)
        };
     
    \addplot[
        dashed,
        every mark/.append style={solid, fill=white},
        mark=*,
        legend entry={\tiny $\IN_n(A)$}
        ]
        coordinates {
        (5,0.952380952380952)(6,1.00000000000000)(7,1.00000000000000)(8,1.00000000000000)(9,1.00000000000000)(10,1.00000000000000)
        };
     
    \addplot[
        dashed,
        every mark/.append style={solid, fill=white},
        mark=square*,
        legend entry={\tiny $\IN_n(L)$}
        ]
        coordinates {
        (5,0.380952380952381)(6,0.508928571428571)(7,0.616647127784291)(8,0.722047314923091)(9,0.849678259537307)(10,0.941936083517951)
        };
    
    \addplot[
        dashed,
        every mark/.append style={solid, fill=white},
        mark=diamond*,
        legend entry={\tiny $\IN_n(Q)$}
        ]
        coordinates {
        (5,0.523809523809524)(6,0.696428571428571)(7,0.726846424384525)(8,0.716200413780696)(9,0.770959092998315)(10,0.860901367814867)
        };
        
    \addplot [
        dashed,
        every mark/.append style={solid, fill=white},
        mark=triangle*,
        legend entry={\tiny $\IN_n(D)$}
        ]
        coordinates {
        (5,0.714285714285714)(6,0.910714285714286)(7,0.978898007033998)(8,0.996671763965098)(9,0.999659108319289)(10,0.999972517556544)
        };

    \end{axis}
\end{tikzpicture}
    \caption{The fraction of graphs on $n$ vertices that have at least one cospectral mate with respect to a certain associated matrix is denoted as $sp$. The fraction of graphs on $n$ vertices with respect to a certain associated matrix that have at least one coinvariant mate is denoted as $in$.}
    \label{fig:spectruminvariantALQD}
\end{figure}

Table \ref{Tab:codeterminantaldistanceLaplacian} presents the number of connected graphs having at least one cospectral or at least one coinvariant mate for the distance Laplacian and the signless distance Laplacian matrices.
Here, $|\mathcal{G}^{sp-in}_n(M)|$ denotes the number of graphs with a cospectral mate which is also a coinvariant mate with respect to the corresponding matrix $M$. Aouchiche and Hansen \cite{ah} 
enumerated cospectral graphs for the distance, distance Laplacian and distance signless Laplacian matrices of all the connected graphs up to 10 vertices. While most of their results are consistent with ours, in Table \ref{Tab:codeterminantaldistanceLaplacian} we obtain 20455 cospectral graphs with 9 vertices with respect to the distance Laplacian matrix, while they reported that there are 19778 of such graphs.

\begin{table}[H]
    \centering
	{
    \begin{tabular}{lcccccc}
		\hline
        $n$ & 5 & 6 & 7 & 8 & 9 & 10\\
        \hline
        $|\mathcal{G}^{sp}_n(D^L)|$  & 0 & 0 & 43 & 745 & 20,455 & 787,851\\
        $|\mathcal{G}^{in}_n(D^L)|$  &   0 &   0 &  18 &   455 & 16,505 & 642,002\\
        $|\mathcal{G}^{sp-in}_n(D^L)|$ &   
        0 & 0 & 14 & 435 & 16,006 & 611,987\\
        $|\mathcal{G}^{sp}_n(D^Q)|$ & 2 & 6 & 38 & 453 & 8,168 & 319,324\\
        $|\mathcal{G}^{in}_n(D^Q)|$ &   2 & 4 & 20 & 259 & 7,444 &  264,955\\
        $|\mathcal{G}^{sp-in}_n(D^Q)|$ & 2 & 4 & 20 & 243 & 6,676 & 255,964\\
        \hline
	\end{tabular}
	}
\caption{Number of connected graphs with a cospectral or a coinvariant mate for the distance Laplacian and the signless distance Laplacian matrices.}
	\label{Tab:codeterminantaldistanceLaplacian}
\end{table}


 Figure~\ref{fig:spectruminvariantQDLDQ} displays the spectral and the invariant uncertainty for $D^L$ and $D^Q$. We also include the spectral uncertainty for $Q$, since according to Table~\ref{Tab:cospectralALQD}, this would be the best invariant for distinguishing graphs using the spectrum. According to our results, the SNF of $D^Q$ performs better than the spectrum for distinguishing graphs for all considered matrices. We should also note that there is no significant improvement when both the spectrum and the SNF are used together, as the parameter $\mathcal{G}^{sp-in}_n(M)$ indicates in Table \ref{Tab:codeterminantaldistanceLaplacian}, thus this has not been added in   Figure~\ref{fig:spectruminvariantQDLDQ}.

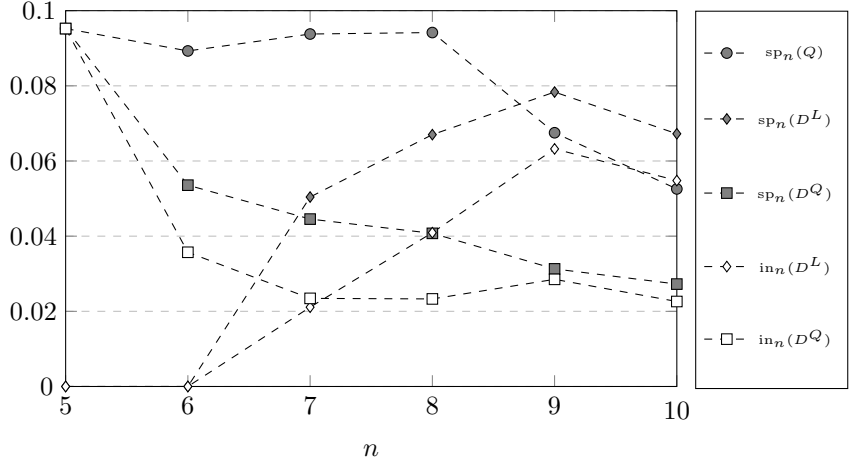
\begin{figure}[H]
    \centering
    \begin{tikzpicture}[trim axis left]
    \begin{axis}[
        scale only axis,
        title={},
        xlabel={$n$},
        width=\textwidth-4cm, 
        height=5cm,
        xmin=5, xmax=10,
        ymin=0.0, ymax=0.1,
        xtick={5,6,7,8,9,10},
        legend style ={ 
            row sep=-0.2cm,
            at={(1.03,1)}, 
            anchor=north west,
            draw=black, 
            fill=white,
            align=left
        },
        ymajorgrids=true,
        grid style=dashed,
        legend columns=1
    ]
     
    \addplot[
        dashed,
        every mark/.append style={solid, fill=gray},
        mark=*,
        legend entry={\tiny $\SP_n(Q)$}
        ]
        coordinates {
        (5,0.0952380952380952)(6,0.0892857142857143)(7,0.0937866354044549)(8,0.0941800845551858)(9,0.0675157039987743)(10,0.0525681959337762)
        };
     
     \addplot [
        dashed,
        every mark/.append style={solid, fill=gray},
        mark=diamond*,
        legend entry={\tiny $\SP_n(D^L)$}
        ]
        coordinates {
        (5,0.000000000000000)(6,0.000000000000000)(7,0.0504103165298945)(8,0.0670144823243681)(9,0.0783476329094530)(10,0.0672424551517675)
        };
    
    \addplot [
        dashed,
        every mark/.append style={solid, fill=gray},
        mark=square*,
        legend entry={\tiny $\SP_n(D^Q)$}
        ]
        coordinates {
        (5,0.0952380952380952)(6,0.0535714285714286)(7,0.0445486518171161)(8,0.0407484033462265)(9,0.0312854297533323)(10,0.0272540489875408)
        };
    
    \addplot [
        dashed,
        every mark/.append style={solid, fill=white},
        mark=diamond*,
        legend entry={\tiny $\IN_n(D^L)$}
        ]
        coordinates {
        (5,0.000000000000000)(6,0.000000000000000)(7,0.0211019929660023)(8,0.0409283079967617)(9,0.0632181706756550)(10,0.0547943592028760)
        };
    
    \addplot [
        dashed,
        every mark/.append style={solid, fill=white},
        mark=square*,
        legend entry={\tiny $\IN_n(D^Q)$}
        ]
        coordinates {
        (5,0.0952380952380952)(6,0.0357142857142857)(7,0.0234466588511137)(8,0.0232976522443105)(9,0.0285123333844032)(10,0.0226136981545198)
        };
    
    \end{axis}
    \end{tikzpicture}
    \caption{The fraction of graphs on $n$ vertices that have at least one cospectral mate with respect to a certain associated matrix is denoted as $sp$. The fraction of graphs on $n$ vertices with respect to a certain associated matrix that have at least one coinvariant mate is denoted as $in$.}
    \label{fig:spectruminvariantQDLDQ}
\end{figure}

In this work we also tested the discrimination power of the $p$-rank on distinguishing graphs. However, since the $p$-rank can take values from $0$ to $n$, in general, it seems not such a good graph invariant. Thus we performed an enumeration of  graphs with the same SNF for the matrices introduced above over $F_p$ with $p\in{2,3,5,7}$. We used this since the SNF of a matrix $M$ over $F_p$ is a finer invariant than the $p$-rank. The enumeration results showed a clear tendency to claim that, for any matrix $M\in\{A,L,Q,D,D^L,D^Q\}$, almost all graphs on $n$ vertices have another graph with the same SNF of $M$ over $F_p$. Thus, we decided not to include these numerical results on the tables.

Aouchiche and Hansen \cite{ah} also explored how to improve the spectral uncertainty by considering two and three matrices together.
Analogously as it was done in \cite{ah}, in Table~\ref{Tab:invariant(sign)distanceLaplacian} we show the number of cospectral and coinvariant graphs when two matrices are considered.
Let $\mathcal{G}^{sp}_n(M, N)$ be the set of graphs in $\mathcal{G}_n$ which have a cospectral mate in $\mathcal{G}_n$ with respect to an associated matrices $M$ and $N$, and let $\mathcal{G}^{in}_n(M, N)$ be the set of graphs in $\mathcal{G}_n$ which have a coinvariant mate in $\mathcal{G}_n$ with respect to an associated matrices $M$ and $N$.
Thus, $\SP_n(M, N)=|\mathcal{G}^{sp}_n(M, N)|/|\mathcal{G}_n|$ and $\IN_n(M, N)=|\mathcal{G}^{in}_n(M, N)|/|\mathcal{G}_n|$.

\begin{table}[H]
	
	\centering
    \begin{tabular}{lcccccc}
		\hline
        $n$ & 7 & 8 & 9 & 10\\
        \hline
        $|\mathcal{G}^{sp}_n(D^L, D^Q)|$ & 0 & 90 & 1,965 & 61,909\\
        $|\mathcal{G}^{in}_n(D^L, D^Q)|$ & 0 & 44 & 1,447 & 46,239\\
        $|\mathcal{G}^{sp}_n(D, D^L)|$ & 0 & 0 & 32 & 9,449 \\
        $|\mathcal{G}^{in}_n(D, D^L)|$ & 0 & 32 & 1,770 & 92,915\\
        $|\mathcal{G}^{sp}_n(D, D^Q)|$ & 0 & 0 & 0 & 7,712 \\
        $|\mathcal{G}^{in}_n(D, D^Q)|$ & 0 & 20 & 432 & 24,517\\
        $|\mathcal{G}^{sp}_n(D, D^L, D^Q)|$ & 0 & 0 & 0 & 7,622 \\
        $|\mathcal{G}^{in}_n(D, D^L, D^Q)|$ & 0 & 0 & 138 & 12,246\\
        \hline
	\end{tabular}
	
\caption{Number of connected graphs with a cospectral and a coinvariant mate for the signless distance Laplacian and distance Laplacian matrices.}
	\label{Tab:invariant(sign)distanceLaplacian}
\end{table}

Figure~\ref{fig:spectruminvarianttwomatrices} shows the spectral and invariant uncertainty of the pairs of matrices obtained in Table~	\ref{Tab:invariant(sign)distanceLaplacian}. 
For the pair $(D^L,D^Q)$, we see and advantage in considering the SNF.
But for the other pairs, we observe that the spectrum is much better.
There is a clear improvement by taking the spectrum of the distance matrix together with the spectrum of either $D^L$ or $D^Q$ to the same obtained by the SNF.
This will also stand in the following analysis.

\begin{figure}[H]
    \centering
    \begin{tikzpicture}[trim axis left]
    \begin{axis}[
        scale only axis,
        title={},
        xlabel={$n$},
        width=\textwidth-4cm, 
        height=5cm,
        xmin=7, xmax=10,
        ymin=0.0, ymax=0.0081,
        xtick={5,6,7,8,9,10},
        legend style ={ 
            row sep=-1.1cm,
            at={(1.03,1)}, 
            anchor=north west,
            draw=black, 
            fill=white,
            align=left
        },
        ymajorgrids=true,
        grid style=dashed,
        legend columns=1
    ]
    
    \addplot [
        dashed,
        every mark/.append style={solid, fill=gray},
        mark=triangle*,
        legend entry={\tiny $\SP_n(D^L, D^Q)$}
        ]
        coordinates {
        (5,0.000000000000000)(6,0.000000000000000)(7,0.000000000000000)(8,0.00809570927408474)(9,0.00752642868086410)(10,0.00528388382573707)
        };
    
    \addplot [
        dashed,
        every mark/.append style={solid, fill=white},
        mark=triangle*,
        legend entry={\tiny $\IN_n(D^L, D^Q)$}
        ]
        coordinates {
        (5,0.000000000000000)(6,0.000000000000000)(7,0.000000000000000)(8,0.00395790231177476)(9,0.00554236249425463)(10,0.00394646181037097)
        };
    
    \addplot [
        dashed,
        every mark/.append style={solid, fill=gray},
        mark=*,
        legend entry={\tiny $\SP_n(D,D^L,)$}
        ]
        coordinates {
        (5,0.000000000000000)(6,0.000000000000000)(7,0.000000000000000)(8,0.000000000000000)(9,0.000122567795311782)(10,0.000806464621773725)
        };
    
    \addplot [
        dashed,
        every mark/.append style={solid, fill=white},
        mark=*,
        legend entry={\tiny $\IN_n(D,D^L)$}
        ]
        coordinates {
        (5,0.000000000000000)(6,0.000000000000000)(7,0.000000000000000)(8,0.00287847440856346)(9,0.00677953117818293)(10,0.00793022122257442)
        };
    
    \addplot [
        dashed,
        every mark/.append style={solid, fill=gray},
        mark=square*,
        legend entry={\tiny $\SP_n(D, D^Q)$}
        ]
        coordinates {
        (5,0.000000000000000)(6,0.000000000000000)(7,0.000000000000000)(8,0.000000000000000)(9,0.000000000000000)(10,0.000658213055679857)
        };
    
    \addplot [
        dashed,
        every mark/.append style={solid, fill=white},
        mark=square*,
        legend entry={\tiny $\IN_n(D, D^Q)$}
        ]
        coordinates {
        (5,0.000000000000000)(6,0.000000000000000)(7,0.000000000000000)(8,0.00179904650535216)(9,0.00165466523670905)(10,0.00209250641676648)
        };
    
    \addplot [
        dashed,
        every mark/.append style={solid, fill=black},
        mark=diamond*,
        legend entry={\tiny $\SP_n(D, D^L,D^Q)$}
        ]
        coordinates {
        (5,0.0)(6,0.0)(7,0.0)(8,0.0)(9,0.0)(10,0.00065053)
        };
        
    \end{axis}
\end{tikzpicture}
    \caption{The fraction of graphs on $n$ vertices that have at least one cospectral mate with respect to a certain associated matrix is denoted as sp. The fraction of graphs on $n$ vertices with respect to a certain associated matrix that have at least one coinvariant mate is denoted as in.}
    \label{fig:spectruminvarianttwomatrices}
\end{figure}
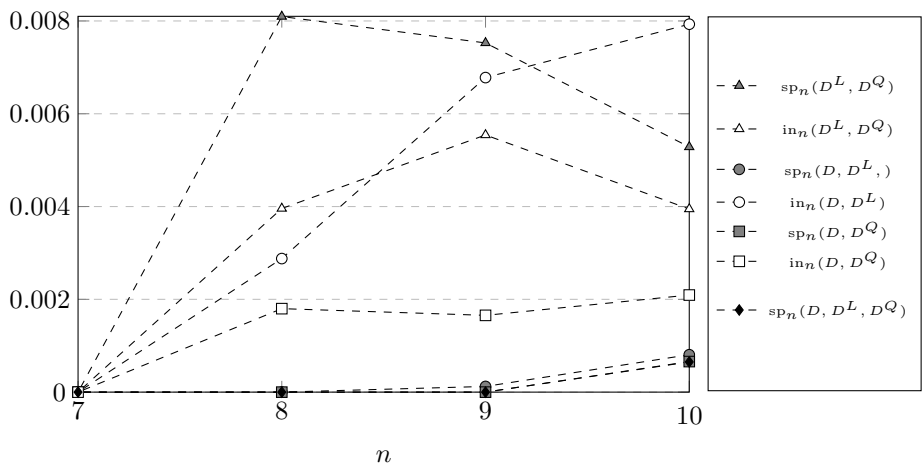

Let $|\mathcal{G}^{sp-in}_n(M,N)|$ denote the number of graphs with a cospectral mate with respect to the matrix $M$ that is also a coinvariant mate with respect to the matrix $N$.
In Table~\ref{Tab:cospectral-invarianttwomatrices}, we compute $|\mathcal{G}^{sp-in}_{10}(M,N)|$ for all possible pairs of associated matrices. The lowest values are $|\mathcal{G}^{sp-in}_{10}(A, D^L)|$, $|\mathcal{G}^{sp-in}_{10}(D, D^L)|$, $|\mathcal{G}^{sp-in}_{10}(A, D^Q)|$ and $|\mathcal{G}^{sp-in}_{10}(A, D^Q)|$.
It is interesting to observe that the combination of the spectrum of $D$ with the SNF of $D^Q$ gives better results than using only the spectrum of the two matrices.
Therefore, this suggests that when distinguishing graphs, we should compute first the SNF of their signless distance Laplacian matrices and then the spectrum of their distance matrices.

\begin{table}[H]
	
	\centering
    \begin{tabular}{|c|cccccc|}
		\hline
        $M \backslash N$ & $A$ & $L$ & $Q$ & $D$ & $D^L$ & $D^Q$\\
        \hline
        $A$   & 2,151,957 & 24,021 & 22,764 & 1,113,103 & 9,253 & 7,688 \\
        $L$   & 521,200 & 1,059,992 & 121,708 & 192,455 & 562,943 & 44,398 \\
        $Q$   & 136,347 & 84,058 & 486,524 & 48,413 & 44,848 & 250,068 \\
        $D$   & 1,073,185 & 15,176 & 13,496 & 1,145,275 & 8,935 & 7,646 \\
        $D^L$ & 300,596 & 563,219 & 52,757 & 110,574 & 611,989 & 47,004 \\
        $D^Q$ & 65,627 & 44,475 & 245,529 & 28,061 & 46,941 & 255,964 \\
        \hline
	\end{tabular}
	
\caption{$|\mathcal{G}^{sp-in}_{10}(M,N)|$}
	\label{Tab:cospectral-invarianttwomatrices}
\end{table}

In order to improve the value obtained for $|\mathcal{G}^{sp}_{10}(D, D^L, D^Q)|$, we explore the use of the following parameters.
Let $|\mathcal{G}^{sp-sp-in}_n(M_1,M_2,M_3)|$ be the number of graphs with a cospectral mate respect to the matrix $M_1$ which is also a cospectral mate with respect $M_2$, and that is also a coinvariant mate with respect to the matrix $M_3$.
Let $|\mathcal{G}^{sp-in-in}_n(M_1,M_2,M_3)|$ be the number of graphs with a cospectral mate respect to the matrix $M_1$ which is also a coinvariant mate with respect $M_2$, and is also a coinvariant mate with respect to the matrix $M_3$.
The results from this analysis are shown in Table~\ref{Tab:cospectral-invarianttwothreematrices}.

\begin{table}[H]
	
	\centering
    \begin{tabular}{lcccccc}
		\hline
        $n$ & 8 & 9 & 10\\
        \hline
        

        
        
        $|\mathcal{G}^{sp-in}_n(A, D^L)|$ & 0 & 32 & 9,253\\
        $|\mathcal{G}^{sp-in}_n(D, D^L)|$ & 0 & 32 & 8,935\\
        $|\mathcal{G}^{sp-in}_n(A, D^Q)|$ & 0 & 2 & 7,688\\
        $|\mathcal{G}^{sp-in}_n(D, D^Q)|$ & 0 & 0 & 7,646\\
        $|\mathcal{G}^{sp-sp-in}_n(A, D, D^L)|$ & 0 & 32 & 8,743\\
        $|\mathcal{G}^{sp-sp-in}_n(A, D, D^Q)|$ & 0 & 0 & 7,550\\
        $|\mathcal{G}^{sp-in-in}_n(A, D^L, D^Q)|$ & 0 & 0 & 7,490\\
        $|\mathcal{G}^{sp-in-in}_n(D, D^L, D^Q)|$ & 0 & 0 & 7,510\\
        \hline
	\end{tabular}

\caption{Number of connected graphs with a cospectral and a coinvariant mate for the signless distance Laplacian and distance Laplacian matrices.}
	\label{Tab:cospectral-invarianttwothreematrices}
\end{table}

From Table~\ref{Tab:cospectral-invarianttwothreematrices}, we can see that  $|\mathcal{G}^{sp-sp-in}_{10}(A, D, D^Q)|$, $|\mathcal{G}^{sp-in-in}_{10}(A, D^L, D^Q)|$ and $|\mathcal{G}^{sp-in-in}_{10}(D, D^L, D^Q)|$ are better than
$|\mathcal{G}^{sp}_{10}(D, D^L, D^Q)|=7,622$. 
Actually, the best performance is obtained with $|\mathcal{G}^{sp-in-in}_{10}(A, D^L, D^Q)|=7,490$.
Note that the order in computing each parameter matters only in the ability of the parameter to distinguish graphs.

To sum up, from the above computational results one can conclude that the best procedure to distinguish graphs using the spectrum and the SNF is first to compute the SNF of their $D^Q$ matrices, since $\IN_n(D^Q)$ has the best performance over the spectral and invariant uncertainty of all matrices. 
Then, if necessary, compute the spectrum of their $A$ matrices, since $|\mathcal{G}_n^{sp-in}(A,D^L)|$ is lower than $|\mathcal{G}_n^{in}(D^L,D^Q)|$ for $n\leq 10$. 
Finally compute the SNF of the $D^L$ matrices.

\subsection{Coinvariant trees}
We end up Section \ref{sec:enumaration} with an observation on coinvariant trees. 
Aouchiche and Hansen \cite{ah2013} reported enumeration results on cospectral trees with at most 20 vertices with respect to $D$, $D^L$ and $D^Q$.
For $D$, they found that among the 123,867 trees on 18 vertices, there are two pairs of $D$-cospectral mates.
Among the 317,955 trees on 19 vertices, there are six pairs of $D$-cospectral mates.
There are 14 pairs of $D$-cospectral mates over all the 823,065 trees on 20 vertices.
And surprisingly, after the enumeration of all 1,346,023 trees on at most 20 vertices, they found no $D^L$-cospectral mates and no $D^Q$-cospectral mates.
This fact lead the authors to conjecture that every tree is determined by its distance Laplacian spectrum, and by its distance signless Laplacian spectrum.

Analogously, for the SNF of $D$, $D^L$ and $D^Q$ of trees one can obtain some similar insights. But for that, first we need to state a result by Hou and Woo \cite{HW2008}, who extended the Graham and Pollak celebrated formula $\det(D(T_{n+1}))=(-1)^nn2^{n-1}$ for any tree $T_{n+1}$ with $n+1$ vertices to the SNF of the distance matrix.

\begin{theorem}\cite[Theorem 3]{HW2008}\label{thm:HouWoo}
Let $T_{n+1}$ be a tree with $n+1$ vertices, then
$\SNF(D(T_{n+1}))=\sf{I}_2\oplus 2\sf{I}_{n-2}\oplus (2n)$.
\end{theorem}

The following is a straight forward consequence from Theorem \ref{thm:HouWoo}, since this implies that all trees on $n$ vertices have the same SNF of its distance matrix $D$.

\begin{corollary}
 All trees with $n$ vertices are $D$-coinvariant mates.
\end{corollary}

After enumerating coinvariant trees with at most 20 vertices with respect to $D^L$ and $D^Q$, we found no $D^L$-coinvariant mates and no $D^Q$-coinvariant mates among all trees with up to 20 vertices.
This fact lead us to conjecture that almost all trees are determined by the SNF of its distance Laplacian, and analogously, by the SNF of its distance signless Laplacian.


\section{Graphs determined by the SNF}\label{sec:graphsdeterminedbytheirdeterminantalideals}

Not much is known regarding graph characterizations using the SNF. A few examples of graphs characterized by the SNF of the adjacency matrix and the Laplacian matrix appear in \cite{alfaval0} or \cite[Chapter~13.8]{bh}. However, our computational results from Section \ref{sec:enumaration} provide an indication that possibly almost no graph has a coinvariant mate when $n\rightarrow \infty$ for the distance Laplacian matrix $D^L$ and the signless distance Laplacian matrix $D^Q$. While the SNF of $D^L$ has been recently used to characterize complete graphs and star graphs  \cite{akm}, to our knowledge there is not yet any graph characterization result using the SNF of $D
^Q$. In this section we present a new method to show that complete graphs can be determined by considering the SNF of $D^Q$.

As mentioned before, it is known that complete graphs and star graphs are determined by the SNF of the distance Laplacian matrix \cite{akm}:

\begin{theorem}\cite{akm}\label{teo:completegraphsaredeterminedbyDL}
Complete graphs are determined by the SNF of the distance Laplacian matrix.
\end{theorem}

\begin{theorem}\cite{akm}\label{teo:stargraphsaredeterminedbyDL}
Star graphs are determined by the SNF of the distance Laplacian matrix.
\end{theorem}

In this section we will show an analogous result to Theorem \ref{teo:completegraphsaredeterminedbyDL}, but instead we will use the SNF of the signless distance Laplacian matrix. In order to do so  we need to define the distance ideals of a graph, which were first introduced in \cite{at}.

Given a connected graph $G=(V,E)$ and a set of indeterminates $X_G=\{x_u \, : \, u\in V(G)\}$, let $\diag(X_G)$ denote the diagonal matrix with the indeterminates in the diagonal and zeroes elsewhere.
The {\it generalized distance matrix} $D(G,X_G)$ of $G$ is the matrix with rows and columns indexed by the vertices of $G$ defined as $\diag(X_G)+D(G)$.
Note that one can recover the distance matrix from the generalized distance matrix by just evaluating $X_G$ at the zero vector, that is, $D(G)=D(G,\bf{0})$.
One can also recover the distance Laplacian and signless distance Laplacian matrices as follows.
Let $tr(u)$ denote the {\it transmission} of a vertex $u$, which is defined as $\sum_{v\in V}d_G(u,v)$, and let $tr(G)$ be the vector whose entries are associated with the transmission of the vertices of $G$.
Then, $D^L(G)=-D(G,-tr(G))$ and $D^Q=D(G,tr(G))$.

Let $\mathbb{Z}[X_G]$ be the polynomial ring over in the variables $X_G$.
For each $k\in[n]:=\{1,..., n\}$, the $k${\it-th distance ideal} $I_k(G,X_G)$ of $G$ is the determinantal ideal given by
\[
\langle {\rm minors}_k(D(G,X_G))\rangle\subseteq \mathbb{Z}[X_G],
\]
where $n$ is the number of vertices of $G$ and ${\rm minors}_i(D(G,X_G))$ is the set of the determinants of the $k\times k$ submatrices of $D(G,X_G)$.

An alternative way of computing the SNF of a matrix $M$ is as follows.
Let $\Delta_k(M)$ denote the {\it greatest common divisor} of the $k$-minors of the matrix $M$.  Then the $k$-{\it th} invariant factor $f_k$ is equal to $\Delta_k(M)/ \Delta_{k-1}(M)$, where $\Delta_0(M)=1$.
Therefore, evaluating the $k${\it -th} distance ideal  $I_k(G,X_G)$ at $X_G={\bf 0}$, we obtain the ideal in $\mathbb{Z}$ generated by $\Delta_k(M)$, that is,  $\langle\Delta_k(D(G))\rangle\subseteq \mathbb{Z}$, from which the invariant factors of $D(G)$ can be recovered \cite{at}.

The following result shows a relation between the SNF of the distance matrix and the distance ideals.

\begin{proposition}\cite{at}\label{teo:eval1}
Let ${\bf d}\in \mathbb{Z}^{V(G)}$.
If $f_1\mid\cdots\mid f_{r}$ are the non-zero invariant factors of the SNF of the matrix $D(G,{\bf d})$, then
\[
I_k(G,{\bf d})=\left\langle \prod_{j=1}^{k} f_j \right\rangle\text{ for all }1\leq k\leq r.
\]
\end{proposition}

In this way, distance ideals can be regarded as a generalization of the SNF of the distance, distance Laplacian and signless distance Laplacian matrices.
Thus, to recover the SNF of $D(G)$, $D^L(G)$ and $D^Q(G)$ from the distance ideals, one just needs to evaluate  $X_G$ at ${\bf 0}$, $-tr(G)$ and $tr(G)$, respectively.

In order to show our main result of this section we need the following theorem, which provides the distance ideals of complete graphs.

\begin{theorem}\cite[Theorem 10]{at}\label{distanceidealscompletegraphs}
    The $k$-th distance ideal of the complete graph $K_n$ with $n$ vertices is generated by
     \[
    \begin{cases}
        \prod_{j=1}^n(x_j-1) + \sum_{i=1}^n\prod_{j\neq i}(x_j-1) & \text{if } k = n,\\
         \left\{\prod_{j\in \mathcal{I}}(x_j-1) : \mathcal{I} \subset [n] \text{ and } |\mathcal{I}|=i-1\right\} & \text{if } k < n.
    \end{cases}
    \]
\end{theorem}

From Theorem \ref{distanceidealscompletegraphs} we can recover the SNFs of the distance, distance Laplacian and signless distance Laplacian matrices of the complete graph. By evaluating the distance ideals at $x_v=0$ for each $v\in V$, we have that $\Delta_i(D(K_n))=1$, for $i\in [n-1]$, and $\Delta_n(D(K_n))=\left|(-1)^n+n(-1)^{n-1}\right|=n-1$.
From which follows that the SNF of the distance matrix of the complete graph with $n$ vertices is ${\sf I}_{n-1}\oplus (n-1)$.
By evaluating the distance ideals at $x_v=-n+1$ for each $v\in V$, we obtain the SNF of distance Laplacian matrix of the complete graph  with $n$ vertices is $1\oplus n{\sf I}_{n-2}\oplus 0$.
And by evaluating the distance ideals at $x_v=n-1$ for each $v\in V$, we obtain the SNF of distance Laplacian matrix of the complete graph  with $n$ vertices is $1\oplus (n-2){\sf I}_{n-2}\oplus 2(n-1)(n-2)$.

An ideal is said to be {\it unit} or {\it trivial} if it is equal to $\langle1\rangle$. Note that every graph with at least one non-loop edge has at least one trivial distance ideal.
Also if the $k$-th invariant factor of $D(G,{\bf d})$ is not equal to $1$, then the ideal $I_k(D, X_D)$ is not trivial.
Therefore, for any graph $G$ and any ${\bf d}\in \mathbb{Z}^{V(G)}$, the number of invariant factors equal to 1 of the matrix $D(G,{\bf d})$ is greater or equal than the number of trivial distance ideals of $G$.
In particular, for any positive integer $k$, the family of graphs with at most $k$ trivial distance ideals contains the family of graphs with at most $k$ invariant factors of $D(G,{\bf d})$ equal to 1.

The last ingredient that we need is the following result. which  characterizes graphs with at most one trivial distance ideal.

\begin{theorem}\cite{at}\label{teo:clasificationofgraphswith1trivialdistance}
A connected graph has only one trivial distance ideal over $\mathbb{Z}[X]$ if and only if $G$ is either a complete graph or a complete bipartite graph.
\end{theorem}

Therefore, if $G$ is a connected graph such that its signless distance Laplacian matrix has at most one invariant factor equal to $1$, then $G$ is a complete graph or a complete bipartite graph.
We have already seen that the SNF of $D^Q(K_n)$ is $1\oplus (n-2){\sf I}_{n-2}\oplus 2(n-1)(n-2)$.
From which follows that complete graphs with at least 4 vertices have one invariant factor of SNF of $D^Q$ equal to one.
Next we prove that these graphs are the only graphs whose signless distance Laplacian matrix have one invariant factor equal to 1.

\begin{theorem}
Let $G$ be a connected graph.
The SNF of $D^Q(G)$ has at most one invariant factor equal to 1 if and only if $G$ is a complete graph with $n\neq 3$ vertices. 
\end{theorem}

\begin{proof}
It only remains to verify that the second invariant factor of complete bipartite graphs is equal to one.
In \cite[Theorem 23]{at}, the second distance ideal of complete bipartite graphs were computed.
Let $D(K_{m,n},\{x_1, \dots, x_m, y_1, \dots, y_n\})$ be the generalized distance matrix of $K_{m,n}$, which is equal to the following matrix
\[
    \begin{bmatrix}
        \diag(x_1,\dots,x_m)-2{\sf I}_m+2{\sf J}_m & {\sf J}_{m,n}\\
        {\sf J}_{n,m} & \diag(y_1,\dots,y_n)-2{\sf I}_n+2{\sf J}_n\\
    \end{bmatrix}.
\]
If $m\geq2$ and $n=1$, then
\[
I_2(K_{m,1},\{x_1, \dots, x_m, y_1\})=\langle x_1-2, \dots, x_m-2,2y_1-1\rangle.
\]
After evaluating the ideal at $x_i=2m-1$, and $y_1=m$, we obtain the ideal $\langle2m-3,2m-1\rangle\subseteq \mathbb{Z}$.
Since the $\gcd(2m-3,2m-1)=1$, it follows that the second invariant factor of $\SNF(D^Q(K_{m,1}))$ is 1.
If $m\geq2$ and $n\geq2$, then
\[
I_2(K_{m,n},\{x_1, \dots, x_m, y_1, \dots, y_n\})=\langle x_1-2, \dots, x_m-2, y_1-2, \dots, y_n-2,3\rangle.
\]
After evaluating the ideal at $x_i=2m+n-2$ and $y_i=2n+m-2$, we obtain the ideal $\langle2m+n-4,m+2n-4,3\rangle\subseteq \mathbb{Z}$.
Since the $\gcd(2m+n-4,m+2n-4,3)=1$, it follows that the second invariant factor of $\SNF(D^Q(K_{m,n}))$ is 1.
\end{proof}

Now we are ready to state the main result of this section.

\begin{corollary}\label{teo:completegraphsaredeterminedbyDQ}
Complete graphs are determined by the SNF of the signless distance Laplacian matrix.
\end{corollary}

It would be interesting to obtain a characterization of graphs whose SNF of $D^L$ and $D^Q$ has two invariant factors equal to 1. 
This could be obtained after showing a characterization of graphs having two trivial distance ideals. However, the latter problem seems to be difficult, since there exist infinitely many minimal induced forbidden graphs \cite{alfaro} (most of them are the same needed in the characterization of the well-known Strong Perfect Graph Theorem).

\section{Concluding remarks}

While the adjacency, Laplacian and signless Laplacian matrices have attracted a lot of attention in the field of spectral characterizations of graphs, for such matrices, the SNF does not seem useful to distinguish graphs, since almost all graphs on 10 vertices have a coinvariant mate. However, our enumeration results suggest that the SNF of the distance Laplacian and the signless distance Laplacian matrices could be a finer invariant to distinguish graphs in cases where other algebraic invariants, such as those derived from the spectrum, fail. This confirms what was suggested by Biggs \cite{biggs1999}. Another argument to consider the SNF as a graph parameter to distinguish graphs is that this is a finer invariant than the $p$-rank: the $p$-rank is just the number of invariant factors not divisible by $p$.

In this work we show that the results by Aouchiche and Hansen \cite{ah} can be pushed further. In particular, we provide numerical evidence that using the invariant factors of the SNF of certain distance matrices one can improve some of the results by Aouchiche and Hansen. In this regard, our computational results suggest that possibly almost no graph has a coinvariant mate when $n\rightarrow \infty$ for the matrices $D^L$ and $D^Q$.


\section*{Acknowledgements}
The authors would like to thank Willem Haemers for a careful reading of the manuscript and for useful comments. The research of A. Abiad is partially supported by the BOF-UGent (Special Research Fund of Ghent University). The research of C. Alfaro is partially supported by SNI and CONACyT.

\end{document}